\documentclass[reqno,11pt]{amsart}
\usepackage{a4wide,color,eucal,enumerate,mathrsfs}
\usepackage[normalem]{ulem}
\usepackage{amsmath,amssymb,epsfig,amsthm} 
\usepackage[latin1]{inputenc}
\usepackage{psfrag}

\def\az{\alpha}

\def\ez{\epsilon}

\def\sz{\sigma}

\def\bint{{\ifinner\rlap{\bf\kern.35em--}
\int\else\rlap{\bf\kern.45em--}\int\fi}\ignorespaces}

\def\bbint{{\ifinner\rlap{\bf\kern.35em--}
\hspace{0.078cm}\int\else\rlap{\bf\kern.45em--}\int\fi}\ignorespaces}

\newcommand{\R}{\mathbb{R}}

\newtheorem{thm}{Theorem}[section]
\newtheorem{prop}[thm]{Proposition}
\newtheorem{cor}[thm]{Corollary}
\newtheorem{defn}[thm]{Definition}
\numberwithin{equation}{section}

\theoremstyle{remark}
\newtheorem{rem}[thm]{Remark}

\def\bint{{\ifinner\rlap{\bf\kern.35em--}
\int\else\rlap{\bf\kern.45em--}\int\fi}\ignorespaces}

\usepackage{graphicx}
\usepackage{import}
\usepackage{xifthen}
\usepackage{pdfpages}
\usepackage{transparent}

\title[Strong  stability of convexity]{Strong stability of convexity with respect to the perimeter}
\author{Alessio Figalli, Yi Ru-Ya Zhang}
\date{\today}

\address{ETH Z\"urich, Department of Mathematics, R\"amistrasse 101, 8092, Z\"urich, Switzerland}
\email{alessio.figalli@math.ethz.ch}  

\address{Academy of Mathematics and Systems Science, the Chinese Academy of Sciences, Beijing 100190, China}
\email{yzhang@amss.ac.cn}

 \thanks{The first author have received funding from the European Research Council under the Grant Agreement No. 
721675 ``Regularity and Stability in Partial Differential Equations (RSPDE)''. The second author is funded by National Key R\&D Program of China (Grant No. 2021YFA1003100),  the Chinese Academy of Science, NSFC grant No. 12288201, and CAS Project for Young Scientists in Basic Research, Grant No. YSBR-031.}


\begin{document}
\begin{abstract}
Let $E\subset \mathbb R^n$, $n\ge 2$, be a set of finite perimeter with $|E|=|B|$, where $B$ denotes the unit ball.
When $n=2$, since convexification decreases perimeter (in the class of open connected sets), it is easy to prove the existence of a convex set $F$, with $|E|=|F|$, such that
$$
 P(E) -  P(F) \ge c\,|E\Delta F|, \qquad c>0. 
$$
Here we prove that, when $n\ge 3$, there exists a convex set $F$, with $|E|=|F|$, such that
$$
 P(E) -  P(F) \ge c(n)  \,f\big(|E\Delta F|\big), \qquad c(n)>0,\qquad f(t)=\frac{t}{|\log t|} \text{ for }t \ll 1.
$$
Moreover, one can choose $F$ to be a small $C^2$-deformation of the unit ball. 
Furthermore, this estimate is essentially sharp as we can show that the inequality above fails for $f(t)=t.$

Interestingly, the proof of our result relies on a new stability estimate for Alexandrov's Theorem on constant mean curvature sets.
\end{abstract}

\maketitle
\section{Introduction}

In shape analysis, perimeter has been frequently used to measure the convexity of a set. For instance, since  for any open connected set $E\subset \mathbb R^2$ the convex hull ${\rm cov}(E)$ is the set with minimal perimeter among those containing $E$, one can use the closeness of the ratio 
$$\frac{P({\rm cov}(E))}{P(E)}$$
to $1$ in order to measure the convexity of $E$; see e.g. \cite{ZR2004} and the references therein. Moreover, via this planar property one can explicitly calculate the difference between the perimeters of the sets and obtain a quantitative inequality: 
\begin{equation}\label{stability plane}
\frac{P(E)}{|E|^{\frac 1 2}}- \frac{P({\rm cov}(E))}{|{\rm cov}(E)|^{\frac 1 2}}\ge \sqrt{\pi} \frac{|{\rm cov}(E)\setminus E|}{|{\rm cov}(E)|^{\frac12}|E|^{\frac12}} \qquad \forall\,E\subset \mathbb R^2 \text{ open, connected, and bounded}.
\end{equation}
 In particular, this implies that, for $|E|=|B|=\pi$,
\begin{equation}\label{stability plane2}
 P(E) -  P(F) \ge    \frac{1}{2\sqrt{2}}\,|E\Delta F|,
\end{equation}
where $F$ is obtained dilating ${\rm cov}(E)$ so that $|F|=|E|$,  see Appendix~\ref{app}. 

Unfortunately, when $n\ge 3$, convexity of a set is not related anymore to the perimeter of its convex hull.
Indeed, there are many open connected  sets  $E\subset \mathbb R^n$ for which  
$$P({\rm cov}(E))>  P(E)$$
(a standard example is a dumbbell-type set with a long handle).
Still, it is reasonable to ask if an inequality similar to \eqref{stability plane} holds in higher dimension with a convex set different from ${\rm cov}(E)$. Namely, given a set $E\subset \mathbb R^n$ with $|E|=|B|$ (here and in the sequel, $B$ denotes the unit ball), we would like to find a set $F$ in the collection $\mathscr C$ of convex sets with $|E|=|F|$ so that
$$P(E)-P(F)\ge c(n)\, f\big(|E\Delta F|\big)$$
where $c(n)>0$, $f\ge 0$ is an increasing function, and $E\Delta F:=(E\cup F)\setminus (E\cap F)$. In addition, one would like to optimize the asymptotic behavior of $f$ near the origin.

\smallskip

When the collection $\mathscr C$ is the class of volume-constrained minimizers for the perimeter, i.e.,\ the class of balls, such an inequality is called \emph {the stability of isoperimetric inequality}. The sharp behavior of $f$ was a long-standing open problem that  was eventually settled by Fusco, Maggi, and Pratelli in \cite{FMP2008}: they proved that, for any
set of finite perimeter $E\subset \mathbb R^n$ with $|E|=|B|$, one has
\begin{equation}\label{FMP}
P(E)-P(B)\ge c(n) \left(\min_{x\in \mathbb R^n} |E\Delta(x+B)|\right)^2.
\end{equation}
In other words, $f(t)=t^\az$ with $\az=2$, and the result is sharp in the sense that such an inequality is false for  $\az<2$.  

This sharp result was extended to the general Wulff inequality in \cite{FMP2010}, which involves a generalization of Euclidean perimeter in the anisotropic case, and the sharp stability with power $f(t)=t^\az,\,\az=2,$ was proven. 
Afterward, some other stronger distances were also considered, see e.g.\ Fusco and Julin \cite{FJ2014} and Neumayer  \cite{N2016}.

Still in this direction, strong stability for a special class of perimeters given by crystalline norms was shown by the authors \cite{FZ2022}. 
There we considered the perimeter  $P_K$ for which the volume-constrained minimizer is given by a convex polyhedron $K$ and we proved that, 
 for any set of finite perimeter $E\subset \mathbb R^n$ close to $K$ with $|E|=|K|$, 
$$P_K(E)-P_K(K')\ge c(n,K)  \biggl(\min_{K'\in \mathscr C(K)}|E\Delta K'|\biggr),$$
where $\mathscr C(K)$ consists of  polyhedra whose faces are parallel to the ones of $K$.
Note that in this case, similarly to \eqref{stability plane}, one has $f(t)=t^\alpha$ with $\alpha=1$.

\smallskip

In this paper we focus on the classical Euclidean perimeter and  investigate the analogue of \eqref{stability plane} in arbitrary  dimension. Our main result in this direction is Theorem~\ref{main thm} below, which shows the strong stability of convexity with respect to the perimeter.  We note that the result is true for $n\geq 2$, although its interest is mainly in the case $n\geq 3$.

Before stating our theorem, let us introduce the following notation: For every continuous function $v\colon \mathbb S^{n-1}\to \mathbb R$ with $\|v\|_{L^\infty(\mathbb S^{n-1})}\le \frac 1 2 $, we define the open set
\begin{equation}
\label{eq:Bv}
B+v=\big\{y\in \mathbb R^n\colon y=t\omega(1+v(\omega))\quad \text{for some }\omega\in \mathbb S^{n-1}, \, 0\le t<1\big\}. 
\end{equation}
Also, we shall write $x+(B+v)$ to denote the translation of the set $B+v$ by a vector $x \in \R^n$.
Now we state our first main result. 
\begin{thm}\label{main thm}
Let $E\subset \mathbb R^n$ be a set of finite perimeter with $|E|=|B|$. 
Define, for $\lambda>0$, 
$$\mathscr C_\lambda:=\left\{F=x+(B+v)\colon v\in C^2(\mathbb S^{n-1}),\, \|v\|_{C^2(\mathbb S^{n-1})}\le \lambda, \,x\in \mathbb R^n, \,|F|=|B| \right\}.$$
Then there exist (small) dimensional constants $\overline\lambda=\overline\lambda(n)>0$ and $c(n)>0$ such that the following holds: For each $0<\lambda\le \overline \lambda$ we can find a uniformly convex set $F\in \mathscr C_\lambda$ satisfying
\begin{equation}\label{main thm ineq}
 P(E) -  P(F) \ge c(n) \lambda  f\big(|E\Delta F|\big), 
 \end{equation}
where 
\begin{equation}\label{def:f}
f(t)=\left\{\begin{array}{ll}
\frac{t}{|\log t|} & \ \text{ when } 0<t<\frac 1 e\\
\frac 1 e  &  \ \text{ when } t\ge \frac 1 e
\end{array}\right. . 
 \end{equation}
\end{thm}
\begin{rem}
Estimate \eqref{main thm ineq} degenerates as $\lambda \to 0$. This is expected since the collection $ \mathscr C_0$ only consists of translations of the unit ball and, in that case,  the sharp result is provided by \eqref{FMP}.
\end{rem}
\begin{rem}
Even if the set $E$ in Theorem~\ref{main thm} belongs to $\mathscr C_\lambda$, it is not necessarily true (at least with the methods used in this paper) that $F=E.$ Namely, $F$ cannot be regarded a projection of $E$ on the class of (smooth uniformly) convex sets. This is different from the results in \cite{FZ2022}. 
\end{rem}

\begin{rem}
As discussed in Section~\ref{not t}  the choice of the function $f$ is essentially sharp for $n \geq 3$, since inequality \eqref{main thm ineq} is false  with $f(t)=t$. 
\end{rem}

To prove Theorem~\ref{main thm}, we will need a stability result for Alexandrov's Theorem in the class of sets that are Lipschitz-close to a sphere.
Let us recall some recent results in this direction.

After a series of very general and fundamental results \cite{CM2017-2,KM2017,DMMN2018} that deal with bubbling phenomena, several authors tried to understand the stability of Alexandrov's Theorem in the absence of bubbling. 

The first fundamental results in this directions are contained in \cite[Theorems 1.8 and 1.10]{KM2017},
where  Krummel and Maggi consider the setting of sets that are Lipschitz-close to a sphere (as in the current paper) and proved sharp stability estimates in terms of the $L^2$-oscillations of the mean curvature.

Later,
Ciraolo and Vezzoni \cite{CV2018} proved that when  $\partial E$ is a connected closed hypersurface of class $C^2$ with small oscillation in the mean curvature $\mathcal H_{\partial E}$, then there exist two concentric balls $B_{r_i}(x_0)$ and $B_{r_e}(x_0)$ such that
\begin{equation}\label{CV}
B_{r_i}(x_0)\subset E\subset B_{r_e}(x_0),\qquad r_e -r_i \le C \left(\max_{x\in \partial E} \mathcal H_{\partial E}(x) -\min_{x\in \partial E} \mathcal H_{\partial E}(x)\right).
\end{equation}
This result is obtained by a refined use of the moving plane method, and the constant $C$ depends on the $C^2$ norm of $\partial E$ (more precisely, on a uniform interior-ball condition).

After this,
Magnanini and Poggesi \cite{MP2020-2} were able to replace the $L^\infty$-oscillation of $\mathcal H_{\partial E}$ in \eqref{CV} with an $L^2$-oscillation, but at the price of a non-sharp power when $n> 3$ (see also  \cite{MP2019,MP2020} for stability estimates with an $L^1$-oscillation). 

\smallskip

In this paper we prove a sharp stability result for Alexandrov's Theorem in the class of sets that are Lipschitz-close to a sphere. This result is similar to \cite[Theorem 1.10]{KM2017} (see also \cite[Theorem 1.3]{MPS2022}). However, our proof provides a control on the difference between an arbitrary constant $\mu$ and $n-1$ when measuring the oscillations of the mean curvature around $\mu$. This difference, although small at first sight, is essential for our application.

It is interesting to notice that Theorem~\ref{mean curvature stability} implies stability estimates for Alexandrov's Theorem that improve the results mentioned above (see Corollary~\ref{cor:Alex} and Remark~\ref{rmk:Alex} below). The experts probably knew these consequences (as they also follow from \cite[Theorem 1.10]{KM2017}), but we could not find them in the literature, so we decided to present them here for the convenience of the reader.

To state the result, we first define the class of nearly spherical sets (recall the notation introduced in \eqref{eq:Bv}).

\begin{defn}
\label{def:spherical}
Let $\sigma>0$ be a small constant.
A set $E\subset \R^n$ is said $\sigma$-nearly spherical
if $E=B+u$ for some function $u:\mathbb S^{n-1} \to \R$ satisfying $\|u\|_{W^{1,\infty}(\mathbb S^{n-1})}\leq \sigma$.
\end{defn}
We also recall that $W^{-1,2}(\mathbb S^{n-1})$ denotes the dual space of $W^{1,2}(\mathbb S^{n-1})$, where
$$\|u\|_{W^{1,2}(\mathbb S^{n-1})}^2:=\|u\|_{L^{2}(\mathbb S^{n-1})}^2+\|D_\tau u\|_{L^{2}(\mathbb S^{n-1})}^2,$$
Here and in the sequel, $D_\tau$ denotes the tangential gradient on the unit sphere. 

Now we can state our stability result of Alexandrov's Theorem for nearly spherical sets. 
\begin{thm}\label{mean curvature stability}
There exists $\sz=\sz(n)>0$ such that the following holds:
let $H=B+w$ be a $\sigma$-nearly spherical set such that
$$0=\int_{H} x\, dx,$$
and let $\mathcal H_{\partial H}$ denote  the (distributional) mean curvature of $\partial H$.
Assume that $\mathcal H_{\partial H} \in L^1(\partial H)$, and define
\begin{equation}
\label{eq:def S}
S(\omega):=\mathcal H_{\partial H}\big(\omega(1+w(\omega))\big)\qquad  \text{for $\mathscr H^{n-1}$-a.e. $\omega \in \mathbb S^{n-1}$}.
\end{equation}
Then, for any $\mu \in \R$,
\begin{equation}
\label{eq:W12 S}
\|w\|_{W^{1,2}(\mathbb S^{n-1})}+|\mu-(n-1)|\le C(n) \|S-\mu\|_{W^{-1,2}(\mathbb S^{n-1})}
\end{equation}
for some dimensional constant $C(n)>0.$ In particular
\begin{equation}
\label{eq:H B}
|H\Delta B|\leq C(n) \inf_{\mu \in \R}\|S-\mu\|_{W^{-1,2}(\mathbb S^{n-1})}.
\end{equation}
Now, assume in addition that $\|w\|_{C^{1,\alpha}} \leq M$ for some $\alpha, M>0$, and that 
$\mathcal H_{\partial H} \in L^p(\partial H)$ for some $p \in \left[\frac{2n-2}{n+1},\infty\right]$ with $p>1$. Then
\begin{equation}
\label{eq:W2p}
\|w\|_{W^{2,p}(\mathbb S^{n-1})} \leq C(n,p,\alpha,M)
\inf_{\mu \in \R} \|\mathcal H_{\partial H}-\mu\|_{L^{p}(\partial H)}\qquad \forall\,p \in \left[\frac{2n-2}{n+1},\infty\right) \text{ with }p>1,
\end{equation}
\begin{equation}
\label{eq:BMO}
\|D^2_{\tau}w\|_{BMO(\mathbb S^{n-1})}  \leq C(n,\alpha,M) \inf_{\mu \in \R}\|\mathcal H_{\partial H}-\mu\|_{L^{\infty}(\partial H)}.
\end{equation}
\end{thm}

\begin{rem}
The assumption $\mathcal H_{\partial H} \in L^1(\partial H)$ is only used to define $S$ in \eqref{eq:def S}.
However this is not strictly need, since for a nearly spherical set $H$ it is possible to define $S \in W^{-1,2}(\mathbb S^{n-1})$ distributionally using formula \eqref{eq:mean}. So, by approximation, \eqref{eq:W12 S} and \eqref{eq:H B} hold for any $\sigma$-nearly spherical set with $S$ defined as in \eqref{eq:mean}, without needing any assumption on  $\mathcal H_{\partial H}$.
\end{rem}

A direct consequence of  Theorem~\ref{mean curvature stability} is the following corollary, which improves the stability results for the Alexandrov's Theorem contained in \cite{CV2018,MP2019,MP2020,MP2020-2}.

\begin{cor}\label{cor:Alex}
Let $p \in (n-1,\infty]$.
There exists a constant $\eta=\eta(n,p)>0$, depending only on $n$ and $p$, such that the following holds:
Let $E\subset \R^n$ satisfy $|E\Delta B|\leq \eta$ and assume that the mean curvature $\mathcal H_{\partial E}$ of $\partial E$ belongs to $L^p(\partial E)$.
Then there exist two concentric balls $B_{r_i}(x_0)$ and $B_{r_e}(x_0)$ such that
$$
B_{r_i}(x_0)\subset E\subset B_{r_e}(x_0),\qquad r_e-r_i \leq C(n,p)\inf_{\mu \in \R}\|\mathcal H_{\partial E} - \mu\|_{L^p(\partial E)},
$$
for some constant $C(n,p)$ depending only on $p$ and $n$.

Moreover, $E-x_0=B+w$ is a nearly spherical set diffeomorphic to a sphere, with
$$
\|w\|_{C^1(\mathbb S^{n-1})}\leq C(n,p)\inf_{\mu \in \R}\|\mathcal H_{\partial E} - \mu\|_{L^p(\partial E)}.
$$
\end{cor}

\begin{rem}
\label{rmk:Alex}
Corollary~\ref{cor:Alex} can be combined with results contained in the literature to obtain new sharp stability estimates for the Alexandrov's Theorem.
For instance, it follows from \cite{DMMN2018} that if
 $$\|\mathcal H_{\partial E}-(n-1)\|_{L^2(\partial E)} \ll 1$$
then $E$ is $L^1$-close to a finite union of unit spheres.
Hence, if $E$ satisfies some additional assumptions that prevent bubbling (e.g., either assuming that $P(E)<2P(B_1)$, or $E$ satisfies an interior cone condition), then we know that $E$ is $L^1$-close to a single sphere. Thus, if $\mathcal H_{\partial E}$ is close to a constant in $L^p$ for some $p>n-1$, Corollary~\ref{cor:Alex} can be applied to obtain a sharp stability estimate. In particular, this allows us to improve the results obtained in \cite{CV2018,MP2019,MP2020,MP2020-2}.

It is worth observing that our method is robust, and can likely be extended to the setting of general smooth elliptic integrand considered in \cite{DMMN2018}. Also, it may be used to improve the convergence results in \cite{FM2011} for liquid drops of small mass lying at equilibrium under the action of  a potential energy.
\end{rem}

\medskip

The paper is organized as follows. In Section~\ref{sect:Alex} we prove our stability results for the Alexandrov's Theorem, namely Theorem~\ref{mean curvature stability} and Corollary~\ref{cor:Alex}. 
 Then, in Section~\ref{sect:nearly} we prove a version  of Theorem~\ref{main thm} in the nearly spherical case, i.e.,  Proposition~\ref{nearly spherical}, and we comment on the optimality of the result.
Finally, the proof of  Theorem~\ref{main thm} is presented in Section~\ref{sect:general}. In the appendix we prove \eqref{stability plane} and \eqref{stability plane2}.

\medskip
{\noindent \bf Notation.}  We often write positive
constants as $C(\cdot)$ and $c(\cdot)$, with
the parentheses including all the parameters on which the constant depends, and simply write $C$ or $c$ if the constant is absolute.  Usually $C(\cdot)$ denotes a constant larger than $1$, and $c(\cdot)$ for a constant less than $1$.
The constant $C(\cdot)$ may
vary between appearances, even within a chain of inequalities.

The Euclidean ball centered at $x$ with radius $r$ is denoted by $B_r(x)$, and  the unit ball centered at the origin is simply denoted by $B$. We denote by $\mathbb S^{n-1}$  the unit sphere,
and by $\mathscr H^{n-1}$ the $(n-1)$-dimensional Hausdorff measure. 
We recall that $D_\tau$ denotes the tangential gradient on the unit sphere. 

\medskip
{\noindent \it Acknowledgments.} The authors are thankful to Francesco Maggi for useful comments on a preliminary version of this paper.

\section{Stability estimates for Alexandrov's Theorem}\label{sect:Alex}
In this section we prove our stability results for Alexandrov's Theorem, namely Theorem~\ref{mean curvature stability} and Corollary \ref{cor:Alex}.

\subsection{Proof of Theorem~\ref{mean curvature stability}}
To simplify the calculation, we define $\xi:=\log(1+w)$ and  note that $\partial H=\left\{\mathbb S^{n-1}\ni \omega \mapsto e^{\xi (\omega)} \omega\right\}$.  Then the mean curvature of $\partial H$ satisfies the following identity in the weak sense:
\begin{equation}\label{eq:mean}
e^\xi  S=  \frac {n-1} {\sqrt{1+|D_\tau \xi|^2}} -{\rm div}_\tau \bigg(\frac{D_\tau \xi}{\sqrt{1+|D_\tau \xi|^2}}\bigg),
\end{equation}
where $S(\omega)=\mathcal H_{\partial H}\big( e^{\xi(\omega)}\omega\big)$ is the mean curvature read on the sphere;
see \cite[Formula (4)]{L2003}\footnote{Note that our definition of mean curvature equals to $(n-1)H$ in \cite{L2003}. Also, the dimension $n$ in \cite{L2003} corresponds to $n-1$ in our setting.} for more details.
Note that, since $\xi \in W^{1,\infty}(\mathbb S^{n-1})$ by assumption, it follows from \eqref{eq:mean} that $S \in W^{-1,2}(\mathbb S^{n-1})$.

Let $\mu\in \R$ be an arbitrary constant, and set $R:=S-\mu$. Then, 
in the current notation, \eqref{eq:mean} becomes
\begin{equation}\label{xi}
\frac {n-1} {\sqrt{1+|D_\tau \xi|^2}} -{\rm div}_\tau \bigg(\frac{D_\tau \xi}{\sqrt{1+|D_\tau \xi|^2}}\bigg) -\mu e^{\xi}- R e^{\xi}=0.
\end{equation}
Hence, testing this equation against $e^{-\xi}$ we obtain
$$
|\mathbb S^{n-1}|\big(\mu-(n-1)\big)=\int_{\mathbb S^{n-1}} \biggl(- \frac{|D_\tau \xi|^2}{\sqrt{1+|D_\tau \xi|^2}}e^{-\xi}+ (n-1)\biggl(\frac {1} {\sqrt{1+|D_\tau \xi|^2}}  e^{-\xi}-1\biggr)  + R\biggr) \,d\mathscr H^{n-1}.
$$
Thus, since $\|\xi\|_{W^{1,\infty}}\le \sz$ and 
$$\left|\int_{\mathbb S^{n-1}} R\,d\mathscr H^{n-1}\right| \leq \|1\|_{W^{1,2}(\mathbb S^{n-1})}\|R\|_{W^{-1,2}(\mathbb S^{n-1})}=C(n)\|R\|_{W^{-1,2}(\mathbb S^{n-1})}, $$
we get
\begin{align*}
\big|\mu-(n-1)\big| &\leq C(n)\Bigl(\|D_\tau \xi\|_{L^2(\mathbb S^{n-1})}^2+ \|R\|_{W^{-1,2}(\mathbb S^{n-1})}\Bigr)\\
&\qquad
+C(n)\int_{\mathbb S^{n-1}} \biggl|\frac {1} {\sqrt{1+|D_\tau \xi|^2}} -1\biggr|e^{-\xi}  \, d\mathscr H^{n-1} +C(n)\biggl| \int_{\mathbb S^{n-1}} \big(e^{-\xi}-1\big) \, d\mathscr H^{n-1}\biggr|
\\
& \leq C(n)\Bigl(\|D_\tau \xi\|_{L^2(\mathbb S^{n-1})}^2+\|\xi\|_{L^2(\mathbb S^{n-1})}^2+ \|R\|_{W^{-1,2}(\mathbb S^{n-1})}\Bigr) +C(n)\biggl| \int_{\mathbb S^{n-1}} \xi \, d\mathscr H^{n-1}\biggr|,
\end{align*} 
where we used that $e^{-\xi}-1=-\xi+O(\xi^2)$.
Also, since $|H|=|B|$, we have 
$$\left|\int_{\mathbb S^{n-1}} \xi \, d\mathscr H^{n-1}\right|  \leq  \left(\frac{n-1}{2} +O(\sz)\right) \int_{\mathbb S^{n-1}} \xi^2 \, d\mathscr H^{n-1},$$
see \cite[Step 1, Theorem 3.1]{F2017}.\footnote{In \cite{F2017} the author shows that $$\label{same volume}\left|\int_{\mathbb S^{n-1}} w \, d\mathscr H^{n-1}\right| \leq  \left(\frac{n-1}{2} +O(\sz)\right) \int_{\mathbb S^{n-1}} w^2 \, d\mathscr H^{n-1}.$$
However, since $\xi=\log(1+w)=w+O(w^2)$, we can replace $w$ with $\xi$.} 
Hence, we have proved that
\begin{equation}
\label{eq:mu}
\big|\mu-(n-1)\big| \leq C(n)\Bigl(\|\xi\|_{W^{1,2}(\mathbb S^{n-1})}^2+ \|R\|_{W^{-1,2}(\mathbb S^{n-1})}\Bigr).
\end{equation}
Testing now \eqref{xi} against $\xi$, and using again that $\|\xi\|_{W^{1,\infty}} \leq \sigma$,
thanks to \eqref{eq:mu} we get
\begin{align*}
&\int_{\mathbb S^{n-1}} \biggl( \frac{|D_\tau  \xi|^2}{\sqrt{1+|D_\tau \xi|^2}} -(n-1)\big(e^\xi-1\big) \xi \biggr)d\mathscr H^{n-1}\\
 &= \int_{\mathbb S^{n-1}} \xi e^{\xi}R\,d\mathscr H^{n-1}  + (n-1) \int_{\mathbb S^{n-1}} \biggl(1-\frac {1} {\sqrt{1+|D_\tau \xi|^2}}\biggr)\xi \,d\mathscr H^{n-1}
+\big(\mu-(n-1)\big)\int_{\mathbb S^{n-1}}\xi e^{\xi}\,d\mathscr H^{n-1}\\
&\leq  \|\xi  e^{\xi}\|_{W^{1,2}(\mathbb S^{n-1})}\|R\|_{W^{-1,2}(\mathbb S^{n-1})}+ C(n)\sigma \|D_\tau \xi \|_{L^2(\mathbb S^{n-1})}^2\\
&\qquad\qquad +C(n)\Bigl( \|\xi\|_{W^{1,2}(\mathbb S^{n-1})}^2+ \|R\|_{W^{-1,2}(\mathbb S^{n-1})}\Bigr)\|\xi \|_{L^1(\mathbb S^{n-1})}\\
& \leq C(n) \|\xi \|_{W^{1,2}(\mathbb S^{n-1})}\|R\|_{W^{-1,2}(\mathbb S^{n-1})}+ C(n)\sigma \|\xi\|_{W^{1,2}(\mathbb S^{n-1})}^2.
\end{align*}
On the other hand, since $\big(e^\xi-1\big) \xi =\xi^2+O(\xi^3)$, we can bound the first term above from below as
\begin{multline*}
\int_{\mathbb S^{n-1}} \biggl( \frac{|D_\tau  \xi|^2}{\sqrt{1+|D_\tau \xi|^2}} -(n-1)\big(e^\xi-1\big) \xi \biggr)d\mathscr H^{n-1}\\
 \geq \sqrt{1-\sigma^2}\int_{\mathbb S^{n-1}}|D_\tau  \xi|^2\,
 d\mathscr H^{n-1} - \big((n-1)+C(n)\sigma\big)\int_{\mathbb S^{n-1}} \xi^2\,
 d\mathscr H^{n-1}.
 \end{multline*}
Hence, combining the two inequalities above, we conclude that
\begin{multline}
\label{eq:Poincare}
(1- C(n)\sigma)\int_{\mathbb S^{n-1}}  {|D_\tau \xi|^2} \,d\mathscr H^{n-1}-  \big(n-1 +C(n)\sz\big) \int_{\mathbb S^{n-1}} \xi^2 \,d\mathscr H^{n-1}\\
\le C(n)  \|\xi \|_{W^{1,2}(\mathbb S^{n-1})}\|R\|_{W^{-1,2}(\mathbb S^{n-1})}.
\end{multline}
We now observe that, since $H$ and $B$ have the same volume and the same barycenter, 
if $\{Y_0,\ldots,Y_n\}$ denote the eigenfunctions corresponding to the first two eigenvalues of the Laplace-Beltrami operator on the sphere, then
$$
\biggl|\int_{\mathbb S^{n-1}} Y_i\,\xi\,d\mathscr H^{n-1}\biggr| \leq C(n)\sigma  \|\xi \|_{L^2(\mathbb S^{n-1})}\qquad \text{for }i=0,\ldots,n,
$$
see e.g. \cite[Step 1, Theorem 3.1]{F2017}.\footnote{As already noted before, in \cite{F2017} these properties are shown for $w$ in place of $\xi$, but since $\xi=w+O(w^2)$, we can replace $w$ with $\xi$.}
 This implies that, for $\sigma \ll1 $, the Poincar\'e constant for $\xi$ is strictly larger than $n-1+\zeta $ for some $\zeta=\zeta(n)>0$, namely
 $$
 \int_{\mathbb S^{n-1}}  {|D_\tau \xi|^2} \,d\mathscr H^{n-1}\geq (n-1+\zeta) \int_{\mathbb S^{n-1}} \xi^2 \,d\mathscr H^{n-1}.
 $$
Thus, combining this inequality with  \eqref{eq:Poincare} we finally obtain
 $$
\|\xi\|_{W^{1,2}(\mathbb S^{n-1})}^2 \le C(n)   \|\xi \|_{W^{1,2}(\mathbb S^{n-1})}\|R\|_{W^{-1,2}(\mathbb S^{n-1})},$$
or equivalently
\begin{equation}
\label{eq:W12 R}
\|\xi\|_{W^{1,2}(\mathbb S^{n-1})} \le C(n) \|R\|_{W^{-1,2}(\mathbb S^{n-1})}.
\end{equation}
Since $\xi:=\log(1+w)$ and $R=S-\mu$ with $\mu$ arbitrary, \eqref{eq:W12 R} and \eqref{eq:mu} imply \eqref{eq:W12 S}.

Finally, since
$$
|H\Delta B|=\int_{\mathbb S^{n-1}}  |(1+ w)^n-1|\, d\mathscr H^{n-1} \leq C(n)\|w\|_{L^1(\mathbb S^{n-1})} \leq C(n)\|w\|_{L^2(\mathbb S^{n-1})},
$$
\eqref{eq:H B} follows from \eqref{eq:W12 S}.

To prove higher regularity, 
we rewrite \eqref{eq:mean} as 
$$
{\rm div}_\tau \bigg(\frac{D_\tau w}{\sqrt{(1+w)^2+|D_\tau w|^2}}\bigg)+(n-1)\biggl(1-\frac {1+w} {\sqrt{(1+w)^2+|D_\tau w|^2}} \biggr) +\mu w=(n-1)-\mu- (1+w)R .
$$
Since $1-\frac {1} {\sqrt{1+|z|^2}}=\int_0^1 \frac{|z|^2}{(1+t|z|^2)^{3/2}}\,dt$,
setting 
$z=\frac{D_\tau w}{1+w}$ we get
$$
1-\frac {1+w} {\sqrt{(1+w)^2+|D_\tau w|^2}}=\int_0^1 \frac{(1+w)|D_\tau w|^2}{((1+w)^2+t|D_\tau w|^2)^{3/2}}\,dt
=\bigg(\int_0^1 \frac{(1+w)D_\tau w}{((1+w)^2+t|D_\tau w|^2)^{3/2}}\,dt\biggr)\cdot D_\tau w
$$
Also, computing the divergence, we have
\begin{multline*}
{\rm div}_\tau \bigg(\frac{D_\tau w}{\sqrt{(1+w)^2+|D_\tau w|^2}}\bigg)
=\frac{\big(1+w)^2+|D_\tau w|^2\big)\Delta_\tau w - D^2_{\tau}w [D_\tau w,D_\tau w]-(1+w)|D_\tau w|^2}{\big((1+w)^2+|D_\tau w|^2\big)^{3/2}}\\
=\frac{\big(1+w)^2+|D_\tau w|^2\big)\Delta_\tau w - D^2_{\tau}w [D_\tau w,D_\tau w]}{\big((1+w)^2+|D_\tau w|^2\big)^{3/2}}
-\frac{(1+w)D_\tau w}{\big((1+w)^2+|D_\tau w|^2\big)^{3/2}}\cdot D_\tau w.
\end{multline*}
Hence, $w$ solves a uniformly elliptic equation of the form
\begin{equation}
\label{eq:PDE xi}
A(w,D_\tau w):D^2_{\tau} w +b(w,D_\tau w)\cdot D_\tau w+\mu w=(n-1)-\mu- (1+w)R,
\end{equation}
where the coefficients are smooth in the range of $w$ (recall that $\|w\|_{W^{1,\infty}} \leq \sigma \ll 1$) and given by
$$
A(s,z):=\frac{\big((1+s)^2+|z|^2\big){\rm Id }- z\otimes z }{\big((1+s)^2+|z|^2\big)^{3/2}},
$$
$$
b(s,z):=(n-1)\int_0^1\frac{(1+s)z}{((1+s)^2+t|z|^2)^{3/2}}\,dt -\frac{(1+s)z}{((1+s)^2+|z|^2)^{3/2}}.
$$
These considerations are particularly useful when $\|w\|_{W^{1,\infty}} \leq \sigma \ll 1$ also satisfies $\|w\|_{C^{1,\alpha}} \leq M$ for some $\alpha, M>0$. 
Indeed, in this case $w$ solves a uniformly elliptic equations with H\"older coefficients.
So it follows from Calderon-Zygmund Theory that
$$
\|w\|_{W^{2,p}(\mathbb S^{n-1})} \leq C(n,p,\alpha,M)\Big(\|\xi\|_{L^{1}(\mathbb S^{n-1})}+ \|(n-1)-\mu- (1+w)R \|_{L^p(\mathbb S^{n-1})}\Big)\qquad \forall\,p \in (1,\infty),
$$
$$
\|D^2_{\tau}\xi\|_{BMO(\mathbb S^{n-1})} \leq C(n,\alpha,M)\Big(\|\xi\|_{L^{1}(\mathbb S^{n-1})}+ \|(n-1)-\mu- (1+w)R\|_{L^\infty(\mathbb S^{n-1})}\Big).
$$
Recalling \eqref{eq:W12 R}, we conclude from Sobolev's embedding\footnote{Since $W^{1,2}(\mathbb S^{n-1})\subset L^{\frac{2n-2}{n-3}}(\mathbb S^{n-1})$, we have the dual inclusion $L^{\frac{2n-2}{n+1}}(\mathbb S^{n-1}) \subset W^{-1,2}(\mathbb S^{n-1})$.} that 
\begin{align*}
\|w\|_{W^{2,p}(\mathbb S^{n-1})} &\leq C(n,p,\alpha,M) \big(\|R\|_{L^{p}(\mathbb S^{n-1})} + \|R\|_{W^{-1,2}(\mathbb S^{n-1})}\big)\\
&\leq C(n,p,\alpha,M) \|R\|_{L^{\max\left\{p,\frac{2n-2}{n+1}\right\}}(\mathbb S^{n-1})} \qquad \forall\,p \in (1,\infty),
\end{align*}
$$
\|D^2_{\tau}w\|_{BMO(\mathbb S^{n-1})}  \leq C(n,\alpha,M)\|R\|_{L^\infty(\mathbb S^{n-1})}.
$$
Note now that, since $w$ is uniformly Lipschitz, integrations over $\mathbb S^{n-1}$ or over $\partial H$ are equivalent up to dimensional multiplicative constants. Hence,
the estimates above imply \eqref{eq:W2p} and \eqref{eq:BMO}, concluding the proof of the theorem. \qed

\subsection{Proof of Corollary \ref{cor:Alex}}
Since $E$ is $\eta$-close in $L^1$ to a ball and its mean curvature belongs to $L^p$ with $p>n-1$,
the regularity theory for almost-minimal hypersurfaces (see for instance \cite[Chapter 5]{S1983}) implies that
$E=B+u$ is a nearly spherical set for some function $u$ satisfying $\|u\|_{C^{1,\alpha}(\mathbb S^{n-1})}=o_\eta(1)$  with $\az=\az(n)>0$, where $o_\eta(1)$ denotes a dimensional constant that goes to $0$ as $\eta \to 0$. In particular, if $\eta$ is small enough we can find a translation $x_0$, with $|x_0|=o_\eta(1)$, such that $H:=E-x_0$ is a nearly spherical set of the form $B+w$ with barycenter at the origin, where $\|w\|_{C^{1,\alpha}(\mathbb S^{n-1})}\leq o_\eta(1)$.

Thus, if $\eta=\eta(n,p)$ is sufficiently small, we can apply \eqref{eq:W2p} and Sobolev's embedding 
to deduce that 
$$
\|w\|_{C^1(\mathbb S^{n-1})} \leq C(n,p) \|w\|_{W^{2,p}(\mathbb S^{n-1})}\leq C(n,p)\inf_{\mu\in \R}\|\mathcal H_{\partial H}-\mu\|_{L^{p}(\partial H)}.
$$
Since $B_{1-\|w\|_{L^\infty(\mathbb S^{n-1})}} \subset H\subset B_{1+\|w\|_{L^\infty(\mathbb S^{n-1})}}$, this concludes the proof. \qed

\section{Nearly spherical case}
\label{sect:nearly}
In this section we first prove Theorem~\ref{main thm}  in the case when $E$ is a $\sigma$-nearly spherical set (recall Definition~\ref{def:spherical}). Then we show that our choice of $f$ is essentially optimal by proving that, for $n \geq 3$, $f$ cannot be chosen to be linear near the origin.

\subsection{A version of Theorem~\ref{main thm} in the nearly spherical case }
We have the following proposition, that corresponds to Theorem~\ref{main thm} in the nearly spherical case.

\begin{prop}\label{nearly spherical}
Let $f$ be defined as in \eqref{def:f}.
There exist small dimensional constants $\overline\sz(n)>0$ and $\overline\lambda(n)>0$,
and a dimensional constant $\overline C(n)>1$, such that 
the following holds: Let $\sigma \in (0,\overline\sz(n))$, $\lambda \in (0,\overline\lambda(n))$, and let $E$ be a $\sigma$-nearly spherical set with $|E|=|B|$. Then one can find a point $x_0\in \mathbb R^n$ and a uniformly convex set of class $C^2$ of the form
$$F=x_0+(B+v),$$ 
with $\|v\|_{C^1(\mathbb S^{n-1})} \leq \overline C(n)\frac{\lambda}{|\log\sigma|}$, $\|D^2_{\tau}v\|_{C^0(\mathbb S^{n-1})} \leq \overline C(n)\lambda$, and $|F|=|E|$, such that
\begin{equation}\label{near sphere}
P(E)-P(F)\ge \lambda  \,f\big(|E\Delta F|\big).
\end{equation}
\end{prop}

\begin{proof}
As before, 
we write the tangential gradient of $u$ as  $D_\tau u$.
Before starting the proof, we observe that
\begin{equation}\label{f'}
f'(t)=\frac{1-\log t }{\log^2 t } \in \big(0, 2 |\log t|^{-1}\big)\quad\text{and} \quad f''(t)=\frac{-2 + \log t }{t (\log  t)^3}>0 \qquad \text{for }t \in (0,1/e).
\end{equation}

\smallskip

\noindent
$\bullet$ {\it Step 1: Solve a suitable variational problem.}
Given $\lambda>0$  small (the smallness to be fixed later), 
consider the variational problem
\begin{equation}\label{variation}
\min\left\{P(G) + 2\lambda f( |E\Delta G|\big)\,:\,|G|=|B| \right\},
\end{equation}
and define $\ell$ to be the minimal value in the problem above.
We claim that a minimizer exists.

Indeed, let $G_j$ be a minimizing sequence for \eqref{variation}. Thanks to the uniform boundedness of $P(G_j)$, up to a subsequence we know that $G_j \to G_\infty$ in $L^1_{\rm loc}(\R^n)$. 

Set $\delta:=\frac{|E|-|G_\infty|}{|E|}\ \in [0,1]$. 
If $\delta=0$ (i.e., $ |G_\infty|=|E|=|B|$), then it follows immediately by the lower semicontinuity of the perimeter that $G_\infty$ is a minimizer. Otherwise, we apply \cite[Lemma 2.2]{FL2015} to write $G_j=H_j\cup R_j$, where $H_j$ and $R_j$ are disjoint sets such that
$$
|H_j \Delta G_\infty| \to 0,\qquad R_j \to \emptyset \quad \text{locally}, \qquad P(H_j)+P(R_j)-P(G_j) \to 0.
$$
Hence $|R_j| \to \delta |E|=\delta |B|$, and it follows from the semicontinuity of the perimeter and the isoperimetric inequality that
\begin{equation}
\label{eq:concavity1}
\begin{split}
n|B|+2\lambda f\big( |E\Delta G_\infty|+\delta|E|\big) &\leq n|B|\left((1-\delta)^{(n-1)/n}+\delta^{(n-1)/n}\right)+2\lambda f\big( |E\Delta G_\infty|+\delta|E|\big)\\
&\leq 
P(G_\infty)+n|B|\delta^{(n-1)/n}+2\lambda f\big( |E\Delta G_\infty|+\delta|E|\big)\\
& \leq \liminf_{j\to \infty}P(H_j)+ n|B|^{1/n}|R_j|^{(n-1)/n}+2\lambda f\big( |E\Delta H_j|+|R_j|\big) \\
&\leq  \liminf_{j\to \infty}P(H_j)+ P(R_j)+2\lambda f\big( |E\Delta H_j|+|R_j|\big)\\
&=  \liminf_{j\to \infty}P(G_j)+2\lambda f\big( |E\Delta G_j|\big) = \ell,
\end{split}
\end{equation}
which shows in particular that
\begin{equation}
\label{eq:concavity12}
P(G_\infty)+n|B|\delta^{(n-1)/n}+2\lambda f\big( |E\Delta G_\infty|+\delta|E|\big)\leq \ell,
\end{equation}
On the other hand, using $B$ as competitor, since $E$ is $\sigma$-nearly spherical we get
$$
\ell \leq n|B|+ 2\lambda f\big( |E\Delta B|\big) \leq n|B|+2\lambda f(C(n)\sigma).
$$
Combining the inequality above with \eqref{eq:concavity1} we deduce that
$$
f\big( |E\Delta G_\infty|+\delta|E|\big) \leq f(C(n)\sigma),
$$
from which it follows (since $f$ is strictly increasing on $[0,1/e]$) that 
\begin{equation}
\label{eq:E delta G}
|E\Delta G_\infty|+\delta|E| \leq C(n) \sigma.
\end{equation}
In particular, this implies that $\delta \leq C(n)\sigma \ll 1.$

Recalling that $\delta>0$, we consider $\frac{1}{(1-\delta)^{1/n}}G_\infty$ as competitor. Then, since $E$ is nearly spherical we get
\begin{equation}
\label{eq:concavity2}\begin{split}
\ell & \leq P\Bigl(\frac{1}{(1-\delta)^{1/n}}G_\infty\Bigr) + 2\lambda f\bigg( \Big|E\Delta \frac{1}{(1-\delta)^{1/n}}G_\infty\Big|\bigg) \\
&=\frac{1}{(1-\delta)^{(n-1)/n}}P(G_\infty) + 2\lambda f\bigg( \frac{1}{1-\delta}\big| (1-\delta)^{1/n}E\Delta G_\infty\big|\bigg)\\
&\leq \frac{1}{(1-\delta)^{(n-1)/n}}P(G_\infty) + 2\lambda f\bigg( \frac{1}{1-\delta}\Big(|E\Delta G_\infty|+C(n)\delta\Big) \bigg).
\end{split}
\end{equation}
 Hence, 
combining \eqref{eq:concavity12} and \eqref{eq:concavity2}, since $f$ is 2-Lipschitz we get
\begin{multline*}
n|B|\delta^{(n-1)/n}\leq \biggl(\frac{1}{(1-\delta)^{(n-1)/n}}-1\biggr)P(G_\infty)
+2\lambda f\bigg( \frac{1}{1-\delta}\Big(|E\Delta G_\infty|+C(n)\delta\Big) \bigg)\\-2\lambda f\big(|E\Delta G_\infty|\big) \leq C(n)\delta,
\end{multline*}
a contradiction if $\sigma>0$ (and therefore $\delta>0$) is sufficiently small.
This proves that $\delta=0$ and therefore $G_\infty$ is a minimizer.

Note that, as a consequence of \eqref{eq:E delta G} applied with $\delta=0$ we deduce that, for any minimizer $G_\infty$,
\begin{equation}
\label{eq:EdeltaG}
|E\Delta G_\infty| \leq C(n) \sigma \ll 1.
\end{equation}

\smallskip

\noindent
$\bullet$ {\it Step 2: Construct the set $F$.}
Let $H$ denote a minimizer of the variational problem \eqref{variation} and note that,
for any set $G\subset \mathbb R^n$ with $|G|=|E|$, since  $f$ is 2-Lipschitz (see \eqref{f'}) we have 
$$P(H)\le P(G) + 2\lambda\big( f( |E\Delta G| ) -f( |E\Delta H| )\big)\le P(G)+4\lambda|G\Delta H|. $$
Also, by \eqref{eq:EdeltaG} applied with $G_\infty=H$, we have
\begin{equation}
\label{eq:EdeltaG 2}
|E\Delta H| \leq C(n) \sigma \ll 1.
\end{equation}
In particular, since $E$ is $\sigma$-nearly spherical it follows that $|H\Delta B| \leq C(n) \sigma$.
Thus, by the theory regularity for almost-minimizers close to a ball (see for instance \cite[Theorem 26.3]{M2012}),  $H=B+w$ for some $\|w\|_{C^{1,\az}(\mathbb S^{n-1})}=o_\sigma(1)$ with $\az=\az(n)>0$, where $o_\sigma(1)$ denotes a dimensional constant that goes to $0$ as $\sigma \to 0$.
Moreover $\partial H$ satisfies the Euler-Lagrange equation
\begin{equation}\label{equ v}
\mathcal{H}_{\partial H} +2\lambda f'\big(|E\Delta H|\big)\left(\mathbf{1}_{\overline H}-\mathbf 1_{\overline E}\right)=\mu \qquad \text{on }\partial H, 
\end{equation}
where $\mathcal H_{\partial H}$ denotes the mean curvature, and $\mu>0$ is a Lagrange multiplier associated to the volume constraint.

Let
$$x_0=\frac{1}{|H|}\int_{H} x\, dx, $$
denote the barycenter of $H$. As $H$ is nearly spherical, then 
$|x_0|=o_\sigma(1)$. Hence, the set $H_{x_0}:=H-x_0$ is nearly spherical,
 has barycenter at the origin, and satisfies
(thanks to  \eqref{equ v})
$$
|\mathcal{H}_{\partial H_{x_0}}-\mu|\leq 2\lambda f'\left(|E\Delta H|\right) \qquad \text{on }\partial H_{x_0}.
$$
Also, $\|w_{x_0}\|_{C^{1,\az}(\mathbb S^{n-1})}=o_\sigma(1) \leq 1$ for $\sigma$ sufficiently small.
Thus, if we write $H_{x_0}=B+w_{x_0}$, we can apply \eqref{eq:W2p} and \eqref{eq:BMO} to deduce that
 \begin{equation}
\label{w-u control w}
\|w_{x_0}\|_{W^{2,n}(\mathbb S^{n-1})}+\|D^2_{\tau}w_{x_0}\|_{BMO(\mathbb S^{n-1})}\leq C(n)\inf_{\mu \in \R}\|\mathcal H_{\partial H}-\mu\|_{L^{\infty}(\partial H)}\leq  C(n)\lambda f'\big(|E\Delta H|\big).
\end{equation}
In particular, Sobolev's inequality, \eqref{f'}, and \eqref{eq:EdeltaG 2}, imply
 \begin{equation}
\label{w C1a}
\|w_{x_0}\|_{C^{1,\alpha}(\mathbb S^{n-1})}\leq C(n)\|w_{x_0}\|_{W^{2,n}(\mathbb S^{n-1})} \leq C(n)\lambda f'\big(|E\Delta H|\big) \leq C(n)\frac{\lambda}{|\log \sigma|},
\end{equation}
while John-Nirenberg's Lemma (see for instance \cite[Chapter 7]{G2014}) yields
\begin{equation}\label{w2p w}
\|D^2_{\tau}w_{x_0}\|_{L^q(\mathbb S^{n-1})}\le C(n)q\|D^2_{\tau}w_{x_0}\|_{BMO(\mathbb S^{n-1})}\leq  C(n)\lambda q f'\big(|E\Delta H|\big) \qquad \forall\,q \in [1,\infty).
\end{equation}
We now define $w_\ez$ as  the solution to the heat equation 
$$\partial_\ez w_\ez = \Delta_{\mathbb S^{n-1}} w_\ez, \qquad w_0=w_{x_0}.$$
Then, denoting by $p_\ez$ the heat kernel on the unit sphere, we have
$$
w_\ez(x)=\int_{\mathbb S^{n-1}} w_{x_0}(y)p_\ez(x,y)\, d\mathscr H^{n-1}(y),
$$
and $p_\ez\geq 0$ satisfies 
\begin{equation}\label{eq:pe}
\|p_\ez(x,\cdot)\|_{L^\infty(\mathbb S^{n-1})}\le C(n) \ez^{\frac{1-n}{2}}, \qquad \|p_\ez(x,\cdot)\|_{L^1(\mathbb S^{n-1})}=1.
\end{equation}
In particular, recalling \eqref{w C1a},
$$
\|w_{\ez}\|_{C^{1,\az}(\mathbb S^{n-1})} \leq \|w_{x_0}\|_{C^{1,\az}(\mathbb S^{n-1})} \leq C(n)\frac{\lambda}{|\log \sigma|}.
$$
Also, it follows from H\"older's inequality, \eqref{w2p w}, and \eqref{eq:pe}, that
\begin{align}
\|D^2_{\tau}w_\ez\|_{L^\infty(\mathbb S^{n-1})}&\le \sup_{x\in \mathbb S^{n-1}} \int_{\mathbb S^{n-1}} |D^2_{\tau}w_{x_0}|(y)p_\ez(x,y)\, d\mathscr H^{n-1}(y)\nonumber\\
&\le \|D^2_{\tau}w_{x_0}\|_{L^q(\mathbb S^{n-1})}\sup_{x\in \mathbb S^{n-1}}  \|p_\ez(x,\cdot)\|_{L^{\frac{q}{q-1}}(\mathbb S^{n-1})}\le C(n)\lambda q f'\big(|E\Delta H|\big) \ez^{\frac {1-n}{2q}} \label{p}
\end{align}
for all $q\geq 1.$ 
By minimizing the last term in \eqref{p} with respect to $q$, which is attained when 
$$c(n) |\log \ez|\le q\le C(n) |\log \ez|,$$ 
we get
$$
\|D^2_{\tau}w_\ez\|_{L^\infty(\mathbb S^{n-1})}\le C(n)\lambda |\log \epsilon| f'\big(|E\Delta H|\big).
$$
Hence, choosing $\ez:= |E\Delta H|^6$
we eventually arrive at 
\begin{equation} \label{weps C2}
\|D^2_{\tau}w_\ez\|_{L^\infty(\mathbb S^{n-1})}\le C(n) \lambda \, \big|\log |E\Delta H|\big| f'\big(|E\Delta H|\big)\le C(n) \lambda,
\end{equation}
where the last inequality follows from \eqref{f'}.
Moreover, using  Poincar\'e inequality, it follows from \eqref{w-u control w} and the theory of the heat flow\footnote{The integration against the heat kernel $p_\ez$ locally corresponds to  integrating $w_{x_0}$ against a convolution kernel of size $\ez^{1/2}$, therefore
$$
\|w_{x_0}-w_\ez\|_{L^2(\mathbb S^{n-1})}
\le C(n) \ez^{1/2} \|D_\tau w_{x_0}\|_{L^2(\mathbb S^{n-1})}.
$$
}
that
\begin{align}
 \|w_{x_0}-w_\ez\|_{L^2(\mathbb S^{n-1})}
\le C(n) \ez^{1/2} \|D_\tau w_{x_0}\|_{L^2(\mathbb S^{n-1})} \leq C(n)\lambda |E\Delta H|^{3}   f'\big(|E\Delta H|\big),\label{L1 difference}
\end{align}
and from \eqref{weps C2}
\begin{align}
 \|D_\tau w_{x_0}-D_\tau w_\ez\|_{L^2(\mathbb S^{n-1})}
\le C(n) \ez^{1/2} \|D^2_{\tau} w_{x_0}\|_{L^2(\mathbb S^{n-1})} \leq C(n) \lambda |E\Delta H|^{3} . \label{L1 difference 2}
\end{align}
Now, let $F'=B+w_\ez$ denote the $\left(C(n)\frac{\lambda}{|\log \sigma|}\right)$-nearly spherical set defined by $w_\ez$.
As a consequence of \eqref{weps C2}, the set $F'$ is uniformly convex provided that $\lambda$ is small enough (the smallness depending only the dimension).
Also, thanks to \eqref{L1 difference},
\begin{multline}   |F'\Delta H_{x_0}| = \int_{\mathbb S^{n-1}}|(1+w_\ez)^n - (1+w_{x_0})^n|\,d\mathscr H^{n-1} \le C(n) \|w_\ez-w_{x_0}\|_{L^1(\mathbb S^{n-1})}\\
\le C(n) \|w_\ez-w_{x_0}\|_{L^2(\mathbb S^{n-1})}\le C(n)\lambda |E\Delta H|^{3}    f'\big(|E\Delta H|\big). \label{F' Hx0}
\end{multline}
Since $|H_{x_0}|=|H|=|E|$, this implies that
$$\left||F'|-|E |\right|\le |F'\Delta H_{x_0}| \leq C(n)\lambda |E\Delta H|^{3}    f'\big(|E\Delta H|\big)$$
namely $E$ and $F'$ have almost the same volume. 
Hence, if we define 
$$\rho:=\biggl(\frac{|E|}{|F'|}\biggr)^{1/n},\qquad F=\rho F'+x_0,$$ 
then
\begin{equation}
\label{eq:rho}
|\rho-1|\leq  C(n)\lambda |E\Delta H|^{3}    f'\big(|E\Delta H|\big),
\end{equation}
and $F=B+v$ is a uniformly convex $\left(C(n)\frac{\lambda}{|\log \sigma|}\right)$-nearly spherical set of class $C^2$. Furthermore, it follows from \eqref{weps C2} that $\|D^2_{\tau} v\|_{C^2}\leq C(n)\lambda$.

\smallskip

\noindent
$\bullet$ {\it Step 3: Estimate the difference of norms.}
We first use $E$ as a competitor in \eqref{variation}  to get
\begin{equation}
\label{eq:P H E}
P(E)- P(H)\ge 2\lambda  f\big(|E\Delta H|\big). 
\end{equation}
Also, recalling the formula for the perimeter of a nearly spherical set $G=B+u$, namely (see e.g.\ \cite[Step 1, Theorem 3.1]{F2017})
\begin{equation}\label{perimeter}
P(G)=\int_{\mathbb S^{n-1}} (1+u)^{n-2} \sqrt{(1+u)^{2(n-1)}+(1+u)^{2(n-2)}|D_\tau u|^2}\, d\mathscr H^{n-1}; 
\end{equation}
it follows from \eqref{L1 difference} and \eqref{L1 difference 2} that
\begin{equation}
\label{eq:P H F}
|P(H_{x_0})-P(F') |\le  C(n) \int_{\mathbb S^{n-1}} \Big(|D_\tau w_{x_0}-D_\tau w_\ez| +  |w_{x_0}-w_\ez| \Big) \, d\mathscr H^{n-1} \leq C(n) \lambda |E\Delta H|^{3}.
\end{equation}
Moreover,  \eqref{eq:rho} implies that
$$
|P(F)-P(F')| \leq \big|\rho^{n-1}-1\big|P(F') \leq C(n)\lambda |E\Delta H|^{3}    f'\big(|E\Delta H|\big)P(F'),
$$
so \eqref{eq:P H F} yields
$$
|P(H)-P(F)|\leq C(n)\lambda |E\Delta H|^{3}    f'\big(|E\Delta H|\big)P(F')+ C(n) \lambda |E\Delta H|^{3}\leq C(n) \lambda |E\Delta H|^{3}.
$$
Hence, combining this bound with \eqref{eq:P H E}, we conclude that
\begin{equation}
\label{eq:P H F2}
P(E)-P(F)=P(E)- P(H)+P(H)-P(F)
\geq 2\lambda \Bigl( f\big(|E\Delta H|\big) - C(n)  |E\Delta H|^{3}\Big).
\end{equation}
Also, since $f$ is $2$-Lipschitz and $f(t) \geq ct^2$ for $t \leq 1$, \eqref{F' Hx0} and \eqref{eq:rho} yield, for $\sigma \ll 1$,
\begin{align*}
f(|E\Delta F|)&\le f(|E\Delta H| ) +2|H_{x_0}\Delta F'|+2|F'\Delta (\rho F')|\le f(|E\Delta H| )+C(n)|E\Delta H|^{3} \\
&\leq \Big(1+C(n)|E\Delta H|\Big) f(|E\Delta H| ) \leq \frac{1+C(n)|E\Delta H|}{1-C(n)|E\Delta H|} \Bigl( f\big(|E\Delta H|\big) - C(n)  |E\Delta H|^{3}\Big) \\
&\leq 2 \Bigl( f\big(|E\Delta H|\big) - C(n)  |E\Delta H|^{3}\Big) .
\end{align*}
Hence, combining this bound with \eqref{eq:P H F2} we finally get 
$$
P(E)-P(F)\geq \lambda f\big(|E\Delta F|\big),
$$
 as desired.
\end{proof}

\subsection{On the sharpness of the function $f$}
\label{not t}
In this section we show that one cannot choose $f(t)=t$ in \eqref{near sphere}, and therefore also in \eqref{main thm ineq}, when $n\ge 3$.

To prove this fact, we begin by recalling that 
 the boundedness of the Laplacian of a function does not imply the boundedness of its Hessian (in other words, $(-\Delta_{\mathbb S^{n-1}})^{-1}$ is not a bounded operator from $L^\infty(\mathbb S^{n-1})$ onto $W^{2,\infty}(\mathbb S^{n-1})$).
 Hence, for any  $\theta \ll 1$ we can find a function $v_\theta:\mathbb S^{n-1} \to \R$ such that
$$
\|v_\theta\|_{W^{1,\,\infty}(\mathbb S^{n-1})}\le 1,\qquad \|\Delta_{\mathbb S^{n-1}} v_\theta\|_{L^\infty(\mathbb S^{n-1})} \leq 1,
\qquad \|D^2_{\tau} v_\theta\|_{L^\infty(\mathbb S^{n-1})} = \theta^{-3}.
$$
In particular, if we define $u_\theta:=\theta^2 v_\theta$, it follows that
\begin{equation}
\label{eq:u theta}
\|u_\theta\|_{W^{1,\,\infty}(\mathbb S^{n-1})}\le \theta^2,\qquad \|\Delta_{\mathbb S^{n-1}} u_\theta\|_{L^\infty(\mathbb S^{n-1})} \leq  \theta^2,
\qquad \|D^2_{\tau} u_\theta\|_{L^\infty(\mathbb S^{n-1})} = \theta^{-1}.
\end{equation}
Furthermore, up to adding to $u_\theta$ a constant of size at most $C(n)\theta^2$, we can ensure that 
$$
  \int_{\mathbb S^{n-1}} (1+u_\theta)^n = n|B|.
$$
Let $E_\theta$ be the $(C(n)\theta^2)$-nearly spherical set defined by $u_{ \theta},$ and note that $|E_\theta|=|B|$.
Also, since the Hessian of $u_\theta$ is large while its Laplacian is small, $E_\theta$ is not convex for $\theta \ll 1$.\footnote{
To see this one can note that, if $\lambda_1(x)\leq \ldots\leq \lambda_{n-1}(x)$ denote the eigenvalues of $D^2_{\tau} u_\theta$, then
$$
\biggl|\sum_{i=1}^{n-1}\lambda_i(x)\biggr| \leq \theta^{2},\qquad \max_{i=1}^{n-1}|\lambda_i(x)|=\theta^{-1}.
$$
This implies that $\lambda_{n-1}(x)\sim\theta^{-1}$ for $\theta \gg 1$, hence $B+u_\theta$ is not convex.
}

Consider now $G=B+w$ a $\theta$-nearly spherical set satisfying $|G|=|E_\theta|$. 
If we define the function $\mathcal F:(-1,\infty)\times \R^{n-1} \to \R$ as 
$$
\mathcal F(s,z):= (1+s)^{n-1}\sqrt{1+\frac{|z|^2}{(1+s)^2}},
$$
recalling \eqref{perimeter} we have
$$
P(E_\theta)-P(G)=\int_{\mathbb S^{n-1}} \Big( \mathcal F(u_\theta,D_\tau u_\theta)-\mathcal F(w,D_\tau w)\Big)\,  d\mathscr H^{n-1}.
$$
Now, since
\begin{equation}
\label{eq:Taylor}
\mathcal F(s,z)= (1+s)^{n-1}\bigg(1+\frac{|z|^2}{2(1+s)^2}\Biggr)+O\big(|z|^4\big)=1+(n-1)s+\frac{(n-1)(n-2)}{2}s^2 +\frac{|z|^2}{2}+O\big(|s|^3+|z|^3\big),
\end{equation}
we see that $\mathcal F$ is convex in a neighborhood of the origin (recall that $n \geq 3$). Hence, if $\theta$ is sufficiently small,
it follows by convexity that 
\begin{align*}
P(E_\theta)-P(G)&\leq \int_{\mathbb S^{n-1}} \Big( \partial_s\mathcal F(u_\theta,D_\tau u_\theta)(u_\theta-w)+ D_z\mathcal F(u_\theta,D_\tau u_\theta)\cdot (D_\tau u_\theta-D_\tau w)\Big)\,  d\mathscr H^{n-1}\\
&=\int_{\mathbb S^{n-1}} \Big( \partial_s\mathcal F(u_\theta,D_\tau u_\theta) -{\rm div}_\tau\big( D_z\mathcal F(u_\theta,D_\tau u_\theta) \big)\Big)(u_\theta-w)\,  d\mathscr H^{n-1}.
\end{align*}
Now, since $\mathcal F$ is smooth near the origin, $\partial_s\mathcal F(0,0)=0$ (see \eqref{eq:Taylor}), and
$\|u_\theta\|_{W^{1,\,\infty}(\mathbb S^{n-1})}\leq C(n)\theta^2$, we get
$$
\big|\partial_s\mathcal F(u_\theta,D_\tau u_\theta)\big|=
\big|\partial_s\mathcal F(u_\theta,D_\tau u_\theta)-\partial_s\mathcal F(0,0)\big|\leq C(n)\|u_\theta\|_{W^{1,\,\infty}(\mathbb S^{n-1})} \leq C(n)\theta^2.
$$
Also,
$$
\big|{\rm div}_\tau\big( D_z\mathcal F(u_\theta,D_\tau u_\theta) \big)\big| \leq \big|D^2_{sz}\mathcal F(u_\theta,D_\tau u_\theta)\cdot D_\tau u_\theta|+\big|D^2_{zz}\mathcal F(u_\theta,D_\tau u_\theta):D^2_{\tau} u_\theta\big|.
$$
The first term in the right hand side above can be bounded by $C(n)\|u_\theta\|_{W^{1,\,\infty}(\mathbb S^{n-1})} \leq C(n)\theta^2$.
For the second term, since $D^2_{zz}\mathcal F(s,z)={\rm Id}+O(|s|+|z|)$ (see \eqref{eq:Taylor}), it follows from \eqref{eq:u theta} that
$$
\big|D^2_{zz}\mathcal F(u_\theta,D_\tau u_\theta):D^2_{\tau} u_\theta\big|\leq |\Delta_{\mathbb S^{n-1}} u_\theta|+C(n)\bigl(|u_\theta|+|D_{\tau} u_\theta|\big)|D^2_{\tau} u_\theta|\leq \theta^2+ C(n)\theta \leq C(n)\theta.
$$
Hence, in conclusion, we proved that
\begin{multline}
\label{eq:Etheta min}
P(E_\theta)-P(G)\leq C(n)\theta \int_{\mathbb S^{n-1}} |u_\theta-w|\,  d\mathscr H^{n-1}\\
\le C(n)   \theta  \int_{\mathbb S^{n-1}}  |(1+ u_\theta)^n-(1+w)^n|\, d\mathscr H^{n-1} \le  C(n) \theta |E_\theta\Delta G|.
\end{multline}
On the other hand, if $\theta$ is sufficiently small, we can apply Proposition~\ref{nearly spherical}
with $\lambda\leq \overline\lambda(n)$, $\sigma=C(n)\theta^2$, and
$E=E_\theta$, to find a $(C(n)\theta^2)$-nearly spherical set convex set $F_\theta$ such that
$$
P(E_\theta)-P(F_\theta)\ge \lambda  f\big(|E_\theta\Delta F_\theta|\big).
$$
Applying \eqref{eq:Etheta min} with $G=F_\theta$ we conclude that
$$
C(n) \theta |E_\theta\Delta F_\theta|\geq \lambda  f\big(|E_\theta\Delta F_\theta|\big).
$$
In particular, if one could choose $f(t)=t$,
taking $\theta<\frac{\lambda}{C(n)}$ we would deduce that $|E_\theta\Delta F_\theta|=0$.
This implies that $E_\theta$ is convex, a contradiction.

\section{General case}
\label{sect:general}
%

\begin{proof}[Proof of Theorem~\ref{main thm}]
Recall that $f(t)$ is 2-Lipschitz (see \eqref{f'}). In particular $f(t)\leq 2t$.

Let $0<\lambda\le  \overline\lambda(n)$, where $\overline\lambda(n)$ is the one in Proposition~\ref{nearly spherical}. 
Suppose that the conclusion of the theorem fails. Then there exists  a sequence of sets of finite perimeter $E_k$, with $|E_k|=|B|$, such that
\begin{equation}\label{assumption}
0<P(E_k)-P(G)< \gamma  \lambda f\big(|E_k\Delta G|\big) \qquad \text{for all $G\in \mathscr C_\lambda$}
\end{equation}
where $\gamma=\gamma(n)>0$ is a small dimensional constant to be fixed later.
Taking $G=B+x$ in \eqref{assumption} with $x\in \R^n$ arbitrary, and using the stability result for the isoperimetric inequality from \cite{FMP2008,FMP2010}, we have
\begin{multline*}
\inf_{x \in \R^n}|E_k\Delta (B+x)|^2\leq C(n)\Big(P(E_k)-P(G)\Big)\\
 \leq C(n)\gamma \lambda \inf_{x \in \R^n} f\big(|E_k\Delta (B+x)|\big)\leq 2 C(n)\gamma \lambda \inf_{x \in \R^n} |E_k\Delta (B+x)|,
\end{multline*}
from which it follows that
$$
\inf_{x \in \R^n}|E_k\Delta (B+x)| \leq 2C(n)\gamma \lambda.
$$
In particular, up to a translation, we can assume that
\begin{equation}
\label{eq:Ek close B}
|E_k\Delta B| \leq 2C(n)\gamma \lambda.
\end{equation}
Consider the minimizer of
$$\min\{ P(F)+  f\big(|E_k\Delta F|\big) \colon |E_k|=|F|\}.$$
Assuming $\gamma$ to be sufficiently small so that $|E_k\Delta B| \ll 1$ (see \eqref{eq:Ek close B}),
the existence of such a minimizer $F_k$ follows from Step 1 in the proof of Proposition~\ref{nearly spherical}. Also, arguing as at the beginning of Step 2 in the proof of Proposition~\ref{nearly spherical}, we deduce that $F_k$ is an almost-minimizer that is $L^1$ close to a ball.
Hence, the regularity theory for almost-minimizers (see for instance \cite[Theorem 26.3]{M2012}) implies that  $F_k=B+w_k$ for some $\|w_k\|_{C^{1,\az}(\mathbb S^{n-1})} \to 0$ as $\gamma \lambda \to 0$. 

In particular, if $\gamma$ is sufficiently small, we can apply Proposition~\ref{nearly spherical} with $\frac{ \lambda }{2\overline C(n)}$ in place of $\lambda$ to deduce the existence of  a convex set $G_k \in \mathscr C_\lambda$ (namely, $G_k=x_k+(B+v_k)$ with $\|v_k\|_{C^2(\mathbb S^{n-1})}\leq \lambda$) such that
\begin{equation}
\label{eq:Fk}
P(F_k)-P(G_k)\ge \frac{ \lambda }{2\overline C(n)}  f\big(|F_k\Delta G_k|\big).
\end{equation}
On the other hand, taking $G=G_k$ in \eqref{assumption} and using $E_k$ as a competitor against the minimality of $F_k$, we get
\begin{equation}
\label{eq:Fk2}
P(F_k)+ f\big(|E_k\Delta F_k|\big)\le P(E_k)\le P(G_k)+   \gamma \lambda f\big(|E_k\Delta G_k|\big).
\end{equation}
Hence, combining \eqref{eq:Fk} and \eqref{eq:Fk2}, we conclude that
\begin{equation}
\label{eq:Fk3}
\frac{ \lambda }{2\overline C(n)}  f\big(|F_k\Delta G_k|\big)+ f\big(|E_k\Delta F_k|\big)\le  \gamma \lambda f\big(|E_k\Delta G_k|\big).
\end{equation}
Note now that $f$ is convex on $(0,1/e)$ (see \eqref{f'}) and $f(t/2)\geq \frac14 f(t)$ for $t>0$ small.
Hence, since $|F_k\Delta G_k|+|E_k\Delta F_k|+|E_k\Delta G_k| \ll 1$, we get
$$f\big(|F_k\Delta G_k|\big)+  f\big(|E_k\Delta F_k|\big)\ge 2 f\left(\frac{|E_k\Delta G_k|}{2}\right) \ge \frac12  f\big(|E_k\Delta G_k|\big).$$
Combining this inequality with \eqref{eq:Fk3} we finally get
$$
\frac12 \min \left\{\frac{ \lambda }{2\overline C(n)},1\right\} \leq \gamma\lambda,
$$
a contradiction if $\gamma=\gamma(n)>0$ is chosen sufficiently small.
\end{proof}

\appendix
\section{Strong stability in 2D}\label{app}

The goal of this short appendix is to prove \eqref{stability plane} and \eqref{stability plane2}. The proof is elementary and goes as follows.
First of all it is well-known that, in two dimensions, convexification decreases the perimeter of every open connected bounded set.
More precisely, if ${\rm cov}(E)$ denotes the convex hull of $E$, then $P({\rm cov}(E))\le P(E).$
Hence, thanks to the inclusion $E\subset {\rm cov}(E)$ we get
\begin{align*}
\frac{P(E)}{|E|^{\frac 1 2}}- \frac{P({\rm cov}(E))}{|{\rm cov}(E)|^{\frac 1 2}} &\geq \frac{P({\rm cov}(E))}{|E|^{\frac 1 2}}- \frac{P({\rm cov}(E)))}{|{\rm cov}(E)|^{\frac 1 2}}\\
&=P({\rm cov}(E)) \frac{|{\rm cov}(E)| - |E|}{|{\rm cov}(E)|^{\frac12}|E|^{\frac12} \left(|{\rm cov}(E)|^{\frac12}+|E|^{\frac12}\right)}\\
&\geq \frac{P({\rm cov}(E))}{|{\rm cov}(E)|^{\frac 1 2}} \frac{|{\rm cov}(E)| - |E|}{2|{\rm cov}(E)|^{\frac12}|E|^{\frac12}} \geq \sqrt{\pi} \frac{|{\rm cov}(E)\setminus E|}{|{\rm cov}(E)|^{\frac12}|E|^{\frac12}},
\end{align*}
where the last inequality follows from the two-dimensional isoperimetric inequality applied to ${\rm cov}(E)$.
This proves \eqref{stability plane}.

Now suppose that $|E|=|B|=\pi$, and assume (up to a translation) that $0\in E$.
Then, we define
$$F=\theta\, {\rm cov}(E),\qquad \theta:=\left(\frac{|E|} {|{\rm cov}(E)|}\right)^{1/2}.$$
Note that $|F|=|E|$ and $F\subset {\rm cov}(E)$ (here we use that $0\in E\subset {\rm cov}(E)$ and that $\theta \leq 1$), therefore
$$
|{\rm cov}(E)\setminus F|=|{\rm cov}(E)|-| F|=|{\rm cov}(E)|-| E|=|{\rm cov}(E)\setminus E|.
$$
Hence, if  $|{\rm cov}(E)|\le 2\pi$,
\begin{align*}
|E\Delta F|&\le |{\rm cov}(E)\setminus E|+|{\rm cov}(E)\setminus F|= 2 |{\rm cov}(E)\setminus E| \\
&\le 2(\sqrt\pi)^{-1}|{\rm cov}(E)|^{\frac 1 2 }(P(E)- P(F)) \leq 2\sqrt{2}(P(E)- P(F)).
\end{align*}
On the other hand, if  $|{\rm cov}(E)|\ge 2\pi,$ then
$$P(E)-P(F)\ge \sqrt\pi \frac{|{\rm cov}(E)\setminus E|}{|{\rm cov}(E)|^{\frac 1 2}}\ge \sqrt\pi \frac{2\pi - \pi}{\sqrt{2\pi}}\ge  \frac{1}{2\sqrt{2}} |E\Delta F|,$$
as $|E\Delta F|\le 2\pi.$
 This proves that, in both cases,
 $$P(E)- P(F)\ge       \frac{1}{2\sqrt{2}} {|E\Delta F|},$$
as desired.
\qed

\end{document}